\documentclass[12pt]{amsart}

\usepackage{palatino}
\usepackage{mathrsfs}
\usepackage{amsfonts}
\usepackage{tikz}
\usepackage{amssymb,bm}
\usepackage{amsmath}
\usepackage{amsthm}
\usepackage[numbers,square]{natbib}

\usepackage[centering, margin={1.0in, 1.5in}, includeheadfoot]{geometry}
\newtheorem{thm}{Theorem}[section]

\newtheorem{lem}[thm]{Lemma}
\newtheorem{cor}[thm]{Corollary}

\newtheorem{rmk}[thm]{Remark}

\newcommand{\Fn}{\mathbb{F}_{q^n}}
\newcommand{\Fs}{\mathbb{F}_{q^s}}
\newcommand{\Fp}{\mathbb{F}_{p}}
\newcommand{\Fpn}{\mathbb{F}_{p^n}}
\newcommand{\F}{\mathbb{F}_{q}}
\newcommand{\Ln}{\mathscr{L}_n(\Fn)}

\DeclareMathOperator{\im}{im}

\title[Compositional inverses, complete mappings, OLS and bent functions]{Compositional inverses, complete mappings, orthogonal Latin squares and bent functions}
\author{Aleksandr Tuxanidy and Qiang Wang}

\address{School of Mathematics and Statistics, Carleton
University,
 1125 Colonel By Drive, Ottawa, Ontario, K1S 5B6,
Canada.} 

\email{AleksandrTuxanidyTor@cmail.carleton.ca, wang@math.carleton.ca}


\thanks{PP (permutation polynomial); CPP (complete permutation polynomials or complete mappings); OLS (orthogonal Latin squares); MOLS (mutually orthogonal Latin squares)}

\thanks{The research of Aleksandr Tuxanidy and Qiang Wang is partially supported by OGS and NSERC, respectively, of Canada.}
\date{\today}
\begin{document}
\font\Bbb msbm10 at 12pt

\begin{abstract}
We study compositional inverses of permutation polynomials, complete mappings, mutually orthogonal Latin squares, and bent vectorial functions.
Recently it was obtained in \cite{wu} the compositional inverses of linearized permutation binomials over finite fields. It was also noted in \cite{tuxanidy} that computing inverses of bijections of  subspaces have applications in determining the compositional inverses of certain permutation classes related to linearized polynomials. In this paper we obtain compositional inverses of a class of linearized binomials permuting the kernel of the trace map.
As an application of this result, we give the compositional inverse of a class of complete mappings. This complete mapping class improves upon a recent construction given in \cite{wu and lin}. We also construct recursively a class of complete mappings involving multi-trace functions.
Finally we use these complete mappings to derive a set of mutually orthogonal Latin squares, and to construct a class of $p$-ary bent 
vectorial functions from the Maiorana-McFarland class.
\end{abstract}

\maketitle

\section{Introduction}

Let $q = p^m$ be the power of a prime number $p$, let $\F$ be a finite field with $q$ elements, and let $\F[x]$ be the ring of polynomials over $\F$. 
We call a polynomial $f \in \F[x]$ a {\em permutation polynomial} (PP) over $\F$ if it induces a permutation of $\F$ under evaluation. 
We denote the {\em composition} of two polynomials $f,g$ by $(f \circ g)(x) := f(g(x))$. 
It is clear that permutation polynomials over $\F$ form a group under composition and subsequent reduction modulo $x^q - x$ that is isomorphic to the symmetric group on $q$ letters. 
Thus for any permutation polynomial $f \in \F[x]$ there exists a unique $f^{-1} \in \F[x]$ such that $f(f^{-1}(x)) \equiv f^{-1}(f(x)) \equiv x \pmod{x^q-x}$.
We call $f^{-1}$ the {\em compositional inverse} of $f$ over $\F$. 


The construction of permutation polynomials over finite fields is an old and difficult subject that continues to attract interest due to their 
 applications in cryptography \cite{rivest,schwenk},
 coding theory \cite{ding13,chapuy}, and combinatorics \cite{ding}. 
 See also \cite{akbary0,akbary,akbary2,charpin,coulter,fernando,hou,hou1,kyureghyan,marcos,mullen,wang13,yuan,yuan1, zha,zieve}, and the references therein for some recent work in the area. However, the problem of 
 determining the compositional inverse of a permutation polynomial seems to be an even more complicated problem. 
 In fact, there are very few known permutation polynomials whose
 explicit compositional inverses have been obtained \cite{coulter1,tuxanidy, wu, wu and liu 2}, and the resulting expressions are usually of a complicated nature except for the classes of the permutation linear polynomials, 
 monomials, Dickson
 polynomials. 
In addition, see \cite{muratovic,wang} for the characterization of the inverse of permutations of $\F$ 
 with form $x^rf(x^s)$ where $s \mid (q-1)$. 

Of particular interest in the study of permutations of finite fields are the linearized polynomials. Polynomials with form $L(x) := \sum_{i=0}^{n-1} a_i x^{q^i}$ are called {\em linearized polynomials} or {\em $q$-polynomials}, 
which are $\F$-linear maps when seen as operators of $\Fn$. Note that $L$ is a permutation polynomial over $\Fn$ if and only if its associate {\em Dickson matrix} given by
 \begin{equation}
  D_L =  \left( \begin{array}{cccc}
                 a_0 & a_1 &\cdots & a_{n-1}\\
                 a_{n-1}^q & a_0^q & \cdots & a_{n-2}^q\\
                 \vdots & \vdots & & \vdots\\
                 a_1^{q^{n-1}} & a_2^{q^{n-1}} & \cdots & a_0^{q^{n-1}}
                \end{array} \right)
\end{equation}
 is non-singular \cite{lidl}. We denote by $\Ln$ the set of all $q$-polynomials over $\Fn$. 
 Recently Wu and Liu obtained in \cite{wu and liu} an expression for the compositional inverse of $L$ in terms of cofactors of $D_L$. Then using this result Wu computed in \cite{wu} the compositional inverses, 
 in explicit form, of arbitrary linearized permutation binomials over finite fields.

More recently, Tuxanidy and Wang showed in \cite{tuxanidy} that the problem of computing the compositional  inverses of certain classes of permutations is equivalent to obtaining the inverses of two other polynomials bijecting subspaces of the finite field, where
 one of these two is a linearized polynomial inducing a bijection
 between kernels of other linearized polynomials. 
 For this they showed in Theorem 2.5 of \cite{tuxanidy} how to obtain linearized polynomials inducing the inverse map over subspaces on which a linearized polynomial induces a bijection. This in fact amounts to
 solving a system of linear equations. Thus, in particular, it is of interest to obtain explicit compositional inverses of linearized permutations of subspaces. 
 
Denote by $T_{q^n|q} : \Fn \rightarrow \F$ the (linearized) {\em trace} map given by
$$
T_{q^n|q}(x) = \sum_{i=0}^{n-1} x^{q^i}.
$$
When it will not cause confusion, we abbreviate this with $T$.
 In Section 2 of this paper we determine a class of linearized binomials permuting the kernel of the trace map and proceed to obtain its inverses on the kernel. 
See Theorem~\ref{thm: InvKer} for more details.

{\em Complete permutation polynomials} (CPP) over $\F$, also called {\em complete mappings}, are permutation polynomials $f \in \F[x]$ such that $f(x) + x$ is also a permutation polynomial over $\F$. 
CPPs have recently become a strong source of interest due to their connection to combinatorial objects such as orthogonal Latin squares \cite{sade, simona}, 
and due to their applications in cryptography; in particular, in the construction of bent functions \cite{ribic, nyberg, simona, stanica}.
See also \cite{tu, wu and lin, wu: even cpp, zhang} and the references therein for some recent work in the area. 
In Section 3 we study complete mappings and give an improvement (Theorem \ref{thm: improv}) to Theorem 3.7 of \cite{wu and lin} by Wu-Lin. This result generalized some earlier corresponding ones found in \cite{chapuy0, simona, wu: even cpp, zhang}.  We also give a recursive construction of complete mappings involving multi-trace functions; see Corollary \ref{cor: recursive}.

As an application of Theorem \ref{thm: InvKer} where we obtained the compositional inverses of linearized binomials permuting the kernel of the trace, we derive in Section 4 the compositional inverse of the class of complete permutation polynomials in Theorem \ref{thm: improv} generalizing some of the classes recently studied in 
 \cite{chapuy0, simona, wu and lin, wu: even cpp, zhang}. Note that since inverses of complete mappings are also complete mappings, Theorem~\ref{thm: CPP inverse} and Corollary~\ref{cor: CPP inverse} imply the construction of a new, if rather complicated, class of of complete permutation polynomials.


In Section 5 we use the new class to construct a set of mutually orthogonal Latin squares. Finally in Section 6 we derive a class of $p$-ary bent vectorial functions
from the Maiorana-McFarland class by the means of our complete mapping.

Before we move on to the following sections let us fix the following notations and definitions. If we view $f \in \F[x]$ as a map of $\F$ and we are given a subset $V$ of $\F$, 
we mean by $f|_{V}$ the map 
obtained by restricting $f$ to $V$, and $f|^{-1}_V$ denotes the inverse map of $f|_{V}$. When the context is clear we may however denote by $f|_{V}^{-1}$ a polynomial
in $\F[x]$ inducing the inverse map of $f|_{V}$. 
When a polynomial $f$ is viewed as a mapping $\F \to \F$, 
we denote by $1/f$ the polynomial $ f^{q-2}$.
Similarly if $x$ is viewed as a point of $\F$, we denote $1/f(x) := f(x)^{q-2}$.
In this case  
we call $f^{-1}(x)$ the {\em preimage}
of $x$ under $f$.

\section{Inverses of linearized binomials permuting kernels of traces}

In this section we study the compositional inverses of binomials permuting the kernel of the trace map.
More precisely, given a positive integer $r < n$, consider the binomial $L_{c,r}(x) := x^{p^r} - cx \in \Fn[x]$, 
where $c \in \mathbb{F}_{q}$. Note that $L_{c,r}(\ker(T_{q^n|q})) \subseteq \ker(T_{q^n|q})$, where $\ker(T_{q^n|q}) = \{ \beta^{q} - \beta \mid \beta \in \Fn \}$ 
is the kernel of the additive map of $T_{q^n|q}$ on $\Fn$.
We would like to discover what are the necessary and sufficient conditions for $L_{c,r}$ to be a permutation of $\ker(T_{q^n|q})$, 
and in such cases obtain a polynomial in $\Fn[x]$ inducing the inverse map of $L_{c,r}|_{\ker(T_{q^n|q})}$.
We only need to consider the case when $L_{c,r}$ permutes $\Fn$ and the case when $L_{c,r}$ permutes $\ker(T_{q^n|q})$ but not $\Fn$.
The former case has already been tackled in \cite{wu} (see Theorem \ref{invLBP} here) and so we focus on the latter case. We give the result in Theorem \ref{thm: InvKer2} of Section 2.1.
In Corollary \ref{cor: InvKer2} we show that under some restrictions of the characteristic, $p$, and the extension degree, $n$, $L_{c,r}$ permutes $\ker(T_{q^n|q})$ for each 
$c \in \F$.
Then in Section 2.2 we explain the method used to obtain the result.

\subsection{Statement and proof of result}

The following result due to Wu gives the compositional inverses of linearized permutation binomials $x^{q^r} - cx$ where $c$ lies in the extension $\Fn$. 
In the case when $c$ lies in $\F$, the binomial is a permutation of $\ker(T_{q^n|q})$ and thus its inverse map over the kernel is clearly the restriction of the inverse over $\Fn$ to $\ker(T_{q^n|q})$.
We state this in Corollary \ref{cor: invLBP}. This accounts for the case when the binomial permuting the kernel of the trace has full rank.
Later in Theorem \ref{thm: InvKer} and \ref{thm: InvKer2} we tackle the case when the linearized binomial permutes the kernel of the trace map but not $\Fn$. 
Denote by $N_{q^n|q} : \Fn \rightarrow \F$ the {\em norm} function given by $N_{q^n|q}(x) = x^{(q^n-1)/(q-1)}$.

\begin{thm}[{\bf Theorem 2.1, \cite{wu}}]\label{invLBP}
Let $c \in \Fn^*$ and let $d:= (n,r)$. Then $L_{c,r}(x) := x^{q^r} - cx \in \Fn[x]$ permutes $\Fn$ if and only if $N_{q^n|q^d}(c) \neq 1$, in which case the compositional inverse
of $L_{c,r}$ over $\Fn$ is given by
\[L_{c,r}^{-1}(x)=\frac{N_{q^n|q^d}(c)}{1 - N_{q^n|q^d}(c)}\sum_{i=0}^{\frac{n}{d}-1}
c^{-\frac{q^{(i+1)r}-1}{q^r-1}}x^{q^{ir}}.\]
\end{thm}

\begin{cor}\label{cor: invLBP}
 Let $q = p^m$ be a power of a prime number $p$, let $n,r,$ be positive integers and denote $d:= (nm,r)$. If $c \in \F$ satisfies $N_{q^n|p^d}(c) \neq 1$, 
 then $L_{c,r}(x) := x^{p^r} - cx \in \Fn[x]$ permutes $\ker(T_{q^n|q})$, having a compositional inverse
 over $\ker(T_{q^n|q})$ given by 
 \[L_{c,r}^{-1}(x)=\frac{N_{q^n|p^d}(c)} {1 - N_{q^n|p^d}(c)}\sum_{i=0}^{\frac{nm}{d}-1}
c^{-\frac{p^{(i+1)r}-1}{p^r-1}}x^{p^{ir}}.\]
\end{cor}

\begin{proof}
Substituting $q$ and $n$ with $p$ and $nm$, respectively, in Theorem \ref{invLBP}, 
we obtain that $L_{c,r}$ permutes $\Fn$. But since $L_{c,r}(\ker(T_{q^n|q})) \subseteq \ker(T_{q^n|q}) \subseteq \Fn$, it follows that $L_{c,r}$ permutes $\ker(T_{q^n|q})$. It now suffices to 
pick $L_{c,r}^{-1}$ as a compositional inverse of $L_{c,r}$ over $\ker(T_{q^n|q})$, which is obtained from Theorem \ref{invLBP} through the aforementioned substitutions. 
\end{proof}

We now focus on the case when the linearized binomial permutes the kernel of the trace map but does not permute $\Fn$. We need the following lemma.

\begin{lem}\label{lem: norm identity}
 Let $n,r,s,$ be positive integers such that $s \mid n$ and $d := (n,r) = (s,r)$. Then the following two identities hold.
 
 (i) $T_{q^{n}|q^s} = T_{q^{nr/d}|q^{sr/d}}|_{\Fn}$;
 
 (ii) $N_{q^s|q^d} = N_{q^{sr/d}|q^r}|_{\Fs}$.
\end{lem}

\begin{proof}
(i)
Since $(n/d, r/d) = 1$ and $n/s \mid n/d$ (implying $(n/s, r/d) = 1$), we have $\{ k \pmod{n/s} \mid 0 \leq k \leq n/s - 1 \} = \{ k r/d \pmod{n/s} \mid 0 \leq k \leq n/s - 1\}$
from which it follows that
$\{ ks \pmod{n} \mid 0 \leq k \leq n/s - 1 \} = \{ k sr/d \pmod{n} \mid 0 \leq k \leq n/s - 1\}$. Hence, for any $\alpha \in \Fn$, we get
$$T_{q^{n}|q^s}(\alpha) := \sum_{k=0}^{\frac{n}{s} - 1} \alpha^{q^{ks}} = \sum_{k=0}^{\frac{n}{s} - 1} \alpha^{q^{ksr/d}} =: T_{q^{nr/d}|q^{sr/d}}(\alpha)$$
as required.

(ii) Since $(s/d, r/d) = 1$, we have $\{ k \pmod{s/d} \mid 0 \leq k \leq s/d - 1 \} = \{ k r/d \pmod{s/d} \mid 0 \leq k \leq s/d - 1\}$, implying 
$\{ kd \pmod{s} \mid 0 \leq k \leq s/d - 1 \} = \{ k r \pmod{s} \mid 0 \leq k \leq s/d - 1\}$. Thus, for any $\beta \in \Fs$, we obtain
$$
N_{q^s|q^d}(\beta) := \beta^{\sum_{k=0}^{\frac{s}{d}-1} q^{kd}} = \beta^{\sum_{k=0}^{\frac{s}{d}-1} q^{kr}} =: N_{q^{sr/d}|q^r}(\beta).
$$
\end{proof}

\begin{thm}\label{thm: InvKer}
 Let $q = p^m$ be a power of a prime number $p$, let $n,r,s$, be positive integers such that $s \mid n$ and $d:= (n,r) = (s,r)$, and let $c \in \Fs$ such that $N_{q^s|q^d}(c) = 1$. 
 Then the binomial $L_{c,r}(x) := x^{q^r} - cx \in \Fn[x]$ induces a permutation of $\ker(T_{q^n|q^s})$ if and only if $p \nmid n/s$. In this case the compositional inverse of $L_{c,r}$ 
 over $\ker(T_{q^n|q^s})$ is given by 
 $$
 L_{c,r}|_{\ker(T_{q^n|q^s})}^{-1}(x) = \sum_{j=0}^{\frac{s}{d}-1} c^{-\frac{q^{(j+1)r} - 1}{q^r - 1}} \left(  \left(\dfrac{n}{s} \right)^{-1} \sum_{k=1}^{\frac{n}{s} - 1} k  x^{q^{\frac{ksr}{d}}}\right)^{q^{jr}}.
 $$
\end{thm}

\begin{proof}
Assume that $L_{c,r}$ does not permute $\Fs$, 
i.e., $N_{q^s|q^d}(c) = 1$. Now suppose on the contrary that $p \mid n/s$. Then for any $k \in \Fs$, we have $T_{q^n|q^s}(k) = kn/s = 0$; 
hence $\Fs \subseteq \ker(T_{q^n|q^s})$. Then noting $L_{c,r}(\Fs) \subseteq \Fs$, we obtain that $L_{c,r}$ permutes $\Fs$, a contradiction. 
Necessarily, if $N_{q^s|q^d}(c) = 1$ and $L_{c,r}$ permutes $\ker(T_{q^n|q^s})$, then $p \nmid n/s$. To show that these are also sufficient conditions for $L_{c,r}$ to permute
$\ker(T_{q^n|q^s})$, it suffices to prove that $L_{c,r}|_{\ker(T_{q^n|q^s})}^{-1}$, given above, induces the inverse of $L_{c,r}|_{\ker(T_{q^n|q^s})}$. 
First observe that $N_{q^{s}|q^d}(c)^{-q^r} = N_{q^{s}|q^d}(c)$ since $d \mid r$. 
Moreover $c^{(q^{sr/d} - 1)/(q^r - 1)} = N_{q^s|q^d}(c) = 1$ by Lemma \ref{lem: norm identity}.
For the sake of brevity denote $R(x) := \left(n/s \right)^{-1} \sum_{k=1}^{\frac{n}{s} - 1} k  x^{q^{ksr/d}}$. Now, for all $x \in \ker(T_{q^n|q^s})$, we have
\begin{equation*}
 \begin{split}
  L_{c,r}\left( L_{c,r}|_{\ker(T_{q^n|q^s})}^{-1}(x)  \right)
 &=
 L_{c,r}|_{\ker(T_{q^n|q^s})}^{-1}(x)^{q^r} - c L_{c,r}|_{\ker(T_{q^n|q^s})}^{-1}(x)\\
 &=
 \sum_{j=0}^{\frac{s}{d}-1} \left( c^{-q^r\left(\frac{q^{(j+1)r} - 1}{q^r - 1}\right)} R(x)^{q^{(j+1)r}} - c^{-\frac{q^{(j+1)r} - 1}{q^r - 1}+1}R(x)^{q^{jr}}  \right)\\
 &=
 \sum_{j=1}^{\frac{s}{d}} c^{-q^r\left(\frac{q^{jr} - 1}{q^r - 1}\right)} R(x)^{q^{jr}} - \sum_{j=0}^{\frac{s}{d}-1} c^{-q^r\left(\frac{q^{jr} - 1}{q^r - 1}\right)}R(x)^{q^{jr}}\\
 &=
 c^{-q^r\left(\frac{q^{\frac{s}{d}r} - 1}{q^r - 1}\right)} R(x)^{q^{\frac{s}{d}r}}  - R(x)\\
 &= 
 R(x)^{q^{\frac{s}{d}r}}  - R(x)\\
 &=
 \left(\dfrac{n}{s} \right)^{-1} \sum_{k=1}^{\frac{n}{s} - 1} k \left( x^{q^{(k+1)\frac{sr}{d}}} - x^{q^{k\frac{sr}{d}}} \right)  \\
 &=
 \left(\dfrac{n}{s} \right)^{-1} \left(\sum_{k=1}^{\frac{n}{s} } (k - 1)  x^{q^{k\frac{sr}{d}}} - \sum_{k=0}^{\frac{n}{s} - 1} k  x^{q^{k\frac{sr}{d}}}  \right)\\
 &=
 \left(\dfrac{n}{s} \right)^{-1} \left( \left(\dfrac{n}{s} - 1\right)x - \sum_{k=1}^{\frac{n}{s} - 1 }  x^{q^{k\frac{sr}{d}}}   \right)\\
 &=
 x - \left(\dfrac{n}{s} \right)^{-1} T_{q^{nr/d}|q^{sr/d}}(x)\\
 &=
 x - \left(\dfrac{n}{s} \right)^{-1} T_{q^{n}|q^{s}}(x)\\
 &=
 x
 \end{split}
 \end{equation*}
 as required.
\end{proof}

By the means of some substitutions we obtain the following equivalent result.

\begin{thm}\label{thm: InvKer2}
 Let $q = p^m$ be a power of a prime number $p$, let $n,r,$ be positive integers such that $d:= (nm,r) = (m,r)$, and let $c \in \F$ such that $N_{q|p^d}(c) = 1$. 
 Then $L_{c,r}(x) := x^{p^r} - cx$ induces a permutation
 of $\ker(T_{q^n|q})$ if and only if $p \nmid n$. In this case the compositional inverse of $L_{c,r}$ over $\ker(T_{q^n|q})$ is given by
 $$
 L_{c,r}|_{\ker(T_{q^n|q})}^{-1}(x) = \sum_{j=0}^{\frac{m}{d}-1} c^{-\frac{p^{(j+1)r} - 1}{p^r - 1}} \left(  n^{-1} \sum_{k=1}^{n - 1} k  x^{p^{\frac{kmr}{d}}}\right)^{p^{jr}}.
 $$
\end{thm}

\begin{proof}
 The result follows from Theorem \ref{thm: InvKer} if we substitute $q,s,n,$ there with $p,m,nm,$ respectively.
\end{proof}

\begin{rmk}
 Theorem \ref{thm: InvKer} and \ref{thm: InvKer2} are equivalent. Indeed, given $x^{q^r} - cx \in \Fn[x]$ satisfying Theorem \ref{thm: InvKer}, i.e., $c \in \Fs$ with $s \mid n$, $(s,r)=(n,r)$ and $q = p^m$,
 we get $x^{q^r} - cx = x^{p^{mr}} - cx \in \mathbb{F}_{q_1^{n/s}}[x]$ where $c \in \mathbb{F}_{q_1}$ with $q_1 := q^s = p^{ms}$, $T_{q^n|q^s} = T_{q_1^{n/s}|q_1}$, 
 and $(ms \cdot n/s,mr) = (mn,mr) = (ms, mr)$, satisfying Theorem \ref{thm: InvKer2}.
 Now the fact that Theorem \ref{thm: InvKer2} follows from Theorem \ref{thm: InvKer} implies the equivalence of the two.
\end{rmk}

\begin{cor}[{\bf Lemma 3.4, \cite{wu and lin}}]\label{cor: wu and lin}
 Let $q = p^m$ be a power of a prime number $p$, let $n,r,$ be positive integers such that $(n,r)=1$, and let $c \in \F$. Then $x^{p^r} - cx$ permutes $\ker(T_{q^n|q})$
 if and only if $c$ belongs to the any of the following two cases.
 
 (i) $N_{q|p^{(m,r)}}(c) = 1$ and $p \nmid n$; 
 
 (ii) $N_{q^n|p^{(m,r)}}(c) \neq 1$.
\end{cor}

\begin{proof}
 Using the fact that $d := (nm,r) = (m,r)$ because $(n,r) = 1$ by assumption, the result (i) follows from Theorem \ref{thm: InvKer2}, while (ii) follows from Corollary \ref{cor: invLBP}. 
\end{proof}

\begin{rmk}
By choosing, say $n = 8$ and $m = r = 2$, it is easy to see that the hypotheses of Theorem \ref{thm: InvKer2} are in fact more general than those of Corollary \ref{cor: wu and lin}.
\end{rmk}

The following corollary shows that under some restrictions of $p$ and $n$, $L_{c,r}$ permutes $\ker(T_{q^n|q})$ for each $c \in \F$. Later on in Section 3 
we will make use of this result in order to construct a class of complete mappings.

\begin{cor}\label{cor: InvKer2}
 Let $q = p^m$ be a power of a prime number $p$, let $n,r,$ be positive integers such that $d:= (nm,r) = (m,r)$, $p \nmid n$, and $(n, p^d - 1) = 1$. Then $L_{c,r}(x) := x^{p^r} - cx$
 induces a permutation of $\ker(T_{q^n|q})$ for each $c \in \F$.
\end{cor}

\begin{proof}
 If $N_{q|p^d}(c) = 1$, then $L_{c,r}$ permutes $\ker(T_{q^n|q})$ by Theorem \ref{thm: InvKer2} (because $p \nmid n$ additionally, by assumption). 
 Now assume $N_{q|p^d}(c) \neq 1$. We claim that $N_{q|p^d}(c)^{n} \neq 1$; equivalently, since $N_{q^n|p^d}(c) = N_{q|p^d} \circ N_{q^n|q}(c) = N_{q|p^d}(c)^n \neq 1$, 
 $L_{c,r}$ permutes $\Fn$ and hence $\ker(T_{q^n|q})$ by Corollary \ref{cor: invLBP}. 
 Clearly, if $c = 0$, then $L_{c,r}$ permutes $\ker(T_{q^n|q})$. If $c \neq 0$, denote by $t$ the multiplicative order of $N_{q|p^d}(c) \in \mathbb{F}_{p^d}^*$.
It is clear that $t \mid (p^d - 1)$. On the contrary suppose that $N_{q|p^d}(c)^{n} = 1$. Then $t \mid n$ as well. As a result, $t \mid (n, p^d - 1) = 1$ giving 
$t = 1$, a contradiction to our assumption $N_{q|p^d}(c) \neq 1$. It follows that $L_{c,r}$ permutes $\ker(T_{q^n|q})$ for all $c \in \F$.
\end{proof}

\subsection{Method used to obtain the inverse in Theorem \ref{thm: InvKer}}

We know from Theorem \ref{invLBP} that 
$L_{c,r}(x) = x^{q^r} - cx$ is a permutation polynomial over $\Fn$ if and only if $N_{q^n|q^d}(c) \neq 1$, where $d = (n,r)$. 
In this case $L_{c,r}$ must also permute $\ker(T_{q^n|q^s})$ if $c \in \F$ and thus we may take $L_{c,r}^{-1}$ to be the compositional 
inverse over $\ker(T_{q^n|q^s})$, as done in Corollary \ref{cor: invLBP}. 
In this subsection we consider the case when $L_{c,r}$ permutes $\ker(T_{q^n|q^s})$ but does not permute $\Fn$ (Theorem \ref{thm: InvKer}), and attempt to obtain a compositional inverse over $\ker(T_{q^n|q^s})$.
Our method bears similarity to that employed in \cite{wu} where the compositional inverse of linearized binomials with full rank was obtained.
It consists of modifying the initial problem into an easier one via substitutions of the parameters $q,n,$ and in our case, $s$ as well, and then computing the inverse over
a convenient super space of $\ker(T_{q^n|q^s})$.

If we let $q_1 := q^{d}$, where $d := (n,r) = (s,r)$, then $L_{c,r}(x)$ becomes $x^{q_1^{r/d}} - cx \in \Fn[x] = \mathbb{F}_{q_1^{n/d}}[x]$ 
with $c \in \Fs = \mathbb{F}_{q_1^{s/d}}$, $T_{q^n|q^s} = T_{q_1^{n/d}|q_1^{s/d}}$, $N_{q^n|q^d}(c) = N_{q_1^{n/d}|q_1^{s/d}}(c) = 1$, and $(n/d,r/d) = (s/d,r/d) = 1$.
Thus we first consider
the case when $(n,r) = (s,r) = 1$. Now view $L_{c,r}$ as a polynomial over the composite field $\mathbb{F}_{q^{nr}}$ and observe that $L_{c,r} \in \mathscr{L}_n(\mathbb{F}_{q^{nr}})$, 
i.e., $L_{c,r}$ is a $q^r$-polynomial over $\mathbb{F}_{q^{nr}}$. 
From Lemma \ref{lem: norm identity} we know that $T_{q^{nr}|q^{sr}}|_{\Fn} = T_{q^n|q^s}$ and thus $\ker(T_{q^n|q^s}) = \ker(T_{q^{nr}|q^{sr}}|_{\Fn})
\subseteq \ker(T_{q^{nr}|q^{sr}})$. Moreover $L_{c,r}(\ker(T_{q^{nr}|q^{sr}})) \subseteq \ker(T_{q^{nr}|q^{sr}})$. 
It follows that if $L_{c,r}$ permutes $\ker(T_{q^{nr}|q^{sr}})$, then the inverse map of $L_{c,r}|_{\ker(T_{q^{n}|q^{s}})}$ can be obtained from the inverse of 
$L_{c,r}|_{\ker(T_{q^{nr}|q^{sr}})}$ by restricting its domain to $\Fn$. 
We make use of the following theorem in order to obtain the inverse of $L_{c,r}$ over $\ker(T_{q^{nr}|q^{sr}})$.
Let $v_{q,n} : \mathscr{L}_n(\Fn) \rightarrow \Fn^n$ denote the natural map defined by $v_{q,n}(\sum_{i=0}^{n-1} a_i x^{q^i}) = (a_0 \ a_1 \cdots  a_{n-1})$ (for simplicity we write vectors horizontally).
Recall that a linear operator $L$ on $\F$ is called {\em idempotent} if $L \circ L (x) = L(x)$ for all $x \in \F$. 

\begin{thm}[{\bf Theorem 2.5, \cite{tuxanidy}}]\label{thm: LinearizedInverse}
   Let $V, \bar{V}$, be two equally sized $\F$-subspaces of $\Fn$, let $\varphi \in \Ln$ induce a bijection from $V$ to $\bar{V}$, and let $D_\varphi$
   be the associate Dickson matrix of $\varphi$. Then $L \in \Ln$ induces the inverse map of $\varphi|_V$ if and only if $v_{q,n}(L(x))D_\varphi = v_{q,n}(x - K(x))$
   for some $K \in \Ln$ inducing an idempotent endomorphism of $\Fn$ with $\ker(K) = V$.
  \end{thm}
  
  \begin{rmk}
   The original statement of Theorem 2.5 of \cite{tuxanidy} only gave necessary conditions for $L$ to induce the inverse map of $\varphi|_V$. However it is trivial to show that these
   are also sufficient.
  \end{rmk}

 Since $\ker(T_{q^{nr}|q^{sr}})$ is an $\mathbb{F}_{q^r}$-subspace of $\mathbb{F}_{q^{nr}}$, and $L_{c,r} \in \mathscr{L}_n(\mathbb{F}_{q^{nr}})$ is a $q^r$-polynomial over $\mathbb{F}_{q^{nr}}$ inducing (by assumption)
 a permutation of $\ker(T_{q^{nr}|q^{sr}})$, Theorem \ref{thm: LinearizedInverse} applies.
  Noting that $(n/s)^{-1}T_{q^{nr}|q^{sr}} \in \mathscr{L}_n(\mathbb{F}_{q^{nr}})$ is idempotent on
$\mathbb{F}_{q^{nr}}$ having kernel $\ker(T_{q^{nr}|q^{sr}})$, we know from Theorem \ref{thm: LinearizedInverse} that $L_{c,r}$ permutes $\ker(T_{q^{nr}|q^{sr}})$ (hence $\ker(T_{q^n|q^s})$) if and only if there exists a solution 
$\bar{d} = (d_0,  \ldots, d_{n-1})$
to the linear equation 
\begin{align*}
\bar{d}D_{L_{r}} 
&= 
v_{q^r,n}\left(x - \left(\dfrac{n}{s}\right)^{-1}T_{q^{nr}|q^{sr}}(x)\right)\\
&= -\left( \dfrac{n}{s} \right)^{-1}\left(1 - \dfrac{n}{s}, 0,0, \ldots, 0, 1,0,0, \ldots, 0,1,0,0, \ldots    \right),
\end{align*} 
where the non-zero entries on the right hand side occur at indices (which start at $0$ and end at $n-1$) given by $ks$ with $0 \leq k < n/s$, and
$$
D_{L_{c,r}} = \left( \begin{array}{cccccc}
                 -c & 1 & 0 & 0 &\cdots & 0\\
                 0 & -c^{q^r} & 1 & 0 & \cdots & 0\\
                 \vdots & \vdots &\vdots &\vdots & & \vdots\\
                1 & 0 & 0 & 0 & \cdots & -c^{q^{(n-1)r}}
                \end{array} \right)
$$ 
is the associate Dickson matrix of $L_{c,r} \in \mathscr{L}_n(\mathbb{F}_{q^{nr}})$. If so, it follows that the polynomial 
$\sum_{i=0}^{n-1}d_ix^{q^{ir}}$ induces the inverse map of $L_{c,r}|_{\ker(T_{q^{nr}|q^{sr}})}$, and hence of $L_{c,r}|_{\ker(T_{q^n|q^s})}$. Solving the linear equation we obtain a solution
$$
d_{ks+j} = \left(\dfrac{n}{s} \right)^{-1} k c^{- \frac{q^{(j+1)r} - 1}{q^r-1}},
$$
where $0 \leq j < s$ and $0 \leq k < n/s$. To see this, it suffices to show that $\bar{d}$ satisfies the linear equation, i.e., that
the $(ks+j)$-th entry, $(\bar{d}D_{L_{c,r}})_{ks+j}$, of the $n$-tuple $\bar{d}D_{L_{c,r}}$, satisfies
$$
\left( \bar{d}D_{L_{c,r}}\right)_{ks+j} = 
\begin{cases}
 1 - \left( \dfrac{n}{s} \right)^{-1} &\mbox{ if } j = k = 0;\\
 - \left( \dfrac{n}{s} \right)^{-1} &\mbox{ if } j = 0 \mbox{ and } k \geq 1;\\
 0 &\mbox{ otherwise},
\end{cases}
$$
where $0 \leq j < s$ and $0 \leq k < n/s$. First recall that we have made the assumptions that $N_{q^s|q}(c^{-1}) = 1$ ($d = 1$ since 
$d := (n,r) = (s,r) = 1$ by our assumption above) and $p \nmid n/s$ of Theorem \ref{thm: InvKer}. Then Lemma \ref{lem: norm identity} gives $N_{q^{sr}|q^r}(c^{-1}) = 1$.
We have
\begin{align*}
\left( \bar{d}D_{L_{c,r}} \right)_0 
&= 
-d_0c + d_{n-1} = 0 + d_{\left( \frac{n}{s} - 1  \right)s + s-1}\\
&= 
\left(\dfrac{n}{s}   \right)^{-1}\left( \dfrac{n}{s} - 1  \right) N_{q^{sr}|q^r}\left(c^{-1}  \right)\\
&= 1 - \left( \dfrac{n}{s}  \right)^{-1};\\
\left( \bar{d}D_{L_{c,r}} \right)_{\substack{ks \\ 0 < k < n/s}}
&= 
-d_{ks}c^{q^{ksr}} + d_{ks - 1} = -d_{ks}c + d_{(k-1)s + s-1}\\
&=
\left(\dfrac{n}{s}\right)^{-1}\Bigr[ - k c^{-1} c + \left( k - 1  \right)N_{q^{sr}|q^r}\left( c^{-1}  \right) \Bigr]\\
&= -\left(\dfrac{n}{s} \right)^{-1}; \text{ and}\\
\left( \bar{d}D_{L_{c,r}} \right)_{\substack{ks + j \\ 0 < j < s}}
&= 
-d_{ks+j} c^{q^{(ks+j)r}} + d_{ks+j-1}\\
&= 
\left( \dfrac{n}{s} \right)^{-1}k
\Bigr[ -c^{-\frac{q^{(j+1)r} - 1}{q^r - 1}} c^{q^{jr}} + c^ {-\frac{q^{jr} - 1}{q^r - 1}} \Bigr]\\
&=
\left( \dfrac{n}{s} \right)^{-1}k
\Bigr[ -c^{-\frac{q^{jr} - 1}{q^r - 1}}  + c^ {-\frac{q^{jr} - 1}{q^r - 1}} \Bigr]\\
&= 0,
\end{align*}
as required. As a result, if $(n,r) = 1$, then $N_{q^s|q}(c) = 1$ and $p \nmid n/s$ are sufficient conditions for $L_{c,r}$ to permute $\ker(T_{q^n|q^s})$ but not $\Fn$.
In this case, one of the polynomials inducing the inverse map of $L_{c,r}|_{\ker(T_{q^n|q^s})}$ is given by
\begin{align*}
L_{c,r}|_{\ker(T_{q^n|q^s})}^{-1}(x) 
&= 
\left(\dfrac{n}{s} \right)^{-1} \sum_{j=0}^{s-1} \sum_{k=1}^{\frac{n}{s} - 1} k c^{- \frac{q^{(j+1)r} - 1}{q^r-1}} x^{q^{(ks+j)r}}\\
&=
\sum_{j=0}^{s-1} c^{- \frac{q^{(j+1)r} - 1}{q^r-1}} \left(  \left(\dfrac{n}{s} \right)^{-1} \sum_{k=1}^{\frac{n}{s} - 1} k  x^{q^{ksr}}\right)^{q^{jr}}.
\end{align*}
Now for the general case of $1 \leq r \leq n-1$ and $d = (n,r) = (s,r)$, substitute $q,n,r,s,$ with $q^{d}, n/d, r/d, s/d$, respectively, in the expression above 
for $L_{c,r}|_{\ker(T_{q^n|q^s})}^{-1}$, to obtain the result in Theorem \ref{thm: InvKer}.

\section{Improvement of a class of complete permutation polynomials}

In \cite{wu and lin} Wu and Lin gave the following class of complete permutation polynomials which generalized some of the classes previously studied in 
\cite{chapuy0, simona, wu: even cpp, zhang}.

\begin{thm}[{\bf Theorem 3.7, \cite{wu and lin}}]
 Let $q = p^m$ be a prime power. Let $G \in \F[x]$, let $r$ be a positive integer with $(r,n) = 1$, and assume $p \nmid n$
 and $(n, p^{(m,r)} - 1) = 1$. Then the polynomial
 $$
 x \left( G(T_{q^n|q}(x)) + a T_{q^n|q}(x)^{p^r-1} - ax^{p^{r}-1}    \right)
 $$
 is a complete permutation polynomial over $\Fn$ for each $a \in \F^*$ if and only if $xG(x)$ is a complete permutation
 polynomial over $\F$.
\end{thm}

In this section we improve upon this result by replacing the requirement that $(n,r)=1$ with the more general one of 
$(m,r) = (mn,r)$. See Theorem \ref{thm: improv} below yielding Corollary \ref{cor: improv}.
Finally in Corollary \ref{cor: recursive} we give a recursive construction of complete mappings
involving multi-trace functions. In Section 5 we will use this construction to generate
sets of mutually orthogonal Latin squares, and in Section 6 the class will be of use in deriving a class of vectorial bent functions.
First we need the following lemma due to Coulter-Henderson-Mathews, a consequence of the AGW Criterion given in Lemma 1.2 of \cite{akbary}.

\begin{lem}[{\bf Theorem 3, \cite{coulter}}]\label{lem: coulter}
Let $q = p^m$ be a prime power, let $g \in \mathbb{F}_q[x]$, let $H \in \F[x]$ be a $p$-polynomial, and let $f(x) = H(x) + xg(T(x))$.
Then $f(x)$ is a permutation polynomial over $\mathbb{F}_{q^n}$
if and only if the following two conditions hold.

(i) $\varphi_y(x) := H(x) + xg(y)$ induces a permutation of $\ker(T) = \{\beta^q - \beta  \mid \beta \in \Fn\}$ for each $y \in \F$.

(ii) $\bar{f}(x) := H(x) + xg(x)$ induces a permutation of $\mathbb{F}_q$.
 \end{lem}
 
The following consequence is straightforward.

\begin{cor}\label{cor: coulter}
 Let $q = p^m$ be a prime power and let $g \in \mathbb{F}_q[x]$. Then $xg(T(x))$ is a permutation of $\Fn$ if and only if $xg(x)$ is a permutation of $\F$ and $g(0) \neq 0$.
\end{cor}

The following represents an improvement to Theorem 3.7 in \cite{wu and lin}, replacing the hypothesis of $(n,r) =1$ there with the more general one of $(mn,r) = (m,r)$.

\begin{thm}\label{thm: improv}
 Let $q = p^m$ be a power of a prime number $p$, let $G \in \F[x]$, and let $n,r,$ be positive integers such that $d:= (m,r) = (mn,r)$, $p \nmid n$ and $(n, p^{(m,r)} - 1) = 1$.
Then 
$$
f(x) = ax^{p^r} + x\left( G(T_{q^n|q}(x)) - aT_{q^n|q}(x)^{p^r - 1}\right)
$$
is a complete permutation polynomial over $\Fn$ for each $a \in \F^*$ if and only if $xG(x)$ is a complete permutation polynomial over $\F$.
 \end{thm}
 
 \begin{proof}
  Note that both $f(x)$ and $f(x) + x$ are instances of Lemma \ref{lem: coulter}. It follows that $f$ is a complete permutation polynomial over $\Fn$ if and only if 
  $\varphi_y(x) := ax^{p^r} + x(G(y) - ay^{p^r-1}) = a[x^{p^r} + x(G(y) - ay^{p^r-1})/a]$ is a complete permutation polynomial over $\ker(T_{q^n|q})$, 
  and $\bar{f}(x) := ax^{p^r} + x(G(x) - ax^{p^r-1}) = xG(x)$ is a complete permutation polynomial over $\F$. The former holds for each $a \in \F^*$ 
  by Corollary \ref{cor: InvKer2}, whereas the latter implies the result.
 \end{proof}
 
 Theorem \ref{thm: improv} generalizes \cite[Theorem 3.7]{wu and lin}, \cite[Theorem 3]{chapuy0}, \cite[Theorem 4.1]{simona}, \cite[Theorem 2.1]{wu: even cpp}, and \cite[Theorem 6]{zhang}.
 Letting $G = b \in \F \setminus\{-1, 0\}$ be arbitrary in Theorem \ref{thm: improv}, the following corollary is straightforward.
 
 \begin{cor}\label{cor: improv}
   Let $q = p^m$ be a power of a prime number $p$, and let $n,r,$ be positive integers such that $d := (m,r) = (mn,r)$, $p \nmid n$ and $(n, p^{(m,r)} - 1) = 1$. 
   Then 
   $$
   f(x) = a\left(x^{p^r} - xT_{q^n|q}(x)^{p^r - 1} \right) + bx
   $$
   is a complete permutation polynomial over $\Fn$ for each $a \in \F$ and each $b \in \F \setminus\{-1, 0\}$.
  \end{cor}
  
  Finally we give a recursive construction of CPPs involving multi-trace functions.

 \begin{cor}\label{cor: recursive}
  Let $q = p^m$ be a power of a prime number $p$, let $n,r$ and $1 =: d_0 \mid d_1 \mid \cdots \mid d_n :=n$ be positive integers such that $p \nmid n$,
  $(n,p^{(m,r)} - 1) = 1$ and
  $(d_i m,r) = (d_j m, r)$ for each $0 \leq i,j \leq n$. Let $a_0,\ldots, a_{n-1}$, be such that for $0 \leq k \leq n-1$, $a_k \in \mathbb{F}_{q^{d_k}}$ and $\sum_{l=0}^{k}a_l \neq 0$.
  Let $f_0 \in \F[x]$.
  Then 
  $$
  f(x) = x \left( \sum_{k=0}^{n-1}a_k \left(x^{p^r-1} - T_{q^n|q^{d_k}}(x)^{p^r-1}\right)  + \dfrac{f_0\left(T_{q^n|q}(x)  \right)}{T_{q^n|q}(x)}  \right) 
  $$
  is a complete permutation polynomial over $\Fn$ if and only if $f_0$ is a complete permutation polynomial over $\F$ satisfying $f_0(0) = 0$.
 \end{cor}

\begin{proof}
For the sake of brevity denote $T_{j}^{k} := T_{q^{d_k}|q^{d_j}}$ if $j \leq k$.
 For each $0 \leq i \leq n-1$, let $c_{i} := \sum_{l=0}^{i}a_l \in \mathbb{F}_{q^{d_{i}}}^*$ and recursively define the polynomials 
 $$
 f_{i+1}(x) := x\left( c_i x^{p^r-1} - c_i T^{i+1}_{i}(x)^{p^r-1} + \dfrac{f_i\left( T_{i}^{i+1}(x)   \right)}{T_{i}^{i+1}(x)}  \right) \in \mathbb{F}_{q^{d_{i+1}}}[x].
 $$
 If we substitute $G(x)$, $m$, $n$, $q$, in Theorem \ref{thm: improv}, with $x^{q^{d_i} - 2}f_i(x)$, $d_i m$, $d_{i+1}/d_i$, $q^{d_i}$, respectively, then each $f_{i+1}$ satisfies the conditions of Theorem \ref{thm: improv}. 
 Indeed, we have $(\frac{d_{i+1}}{d_i} \cdot d_i m, r) = (d_{i+1} m, r) = (d_i m, r)$ by assumption; $p \nmid d_{i+1}/d_i$ (because $p \nmid n$) and $(\frac{d_{i+1}}{d_i}, p^{(d_i m, r)} - 1) = (d_{i+1}/d_i, p^{(m,r)} - 1) = 1$
 since $(d_{i+1}/d_i) \mid n$ and $(n, p^{(m,r)} -1 ) = 1$ by assumption as well. Then it follows from Theorem \ref{thm: improv} that $f_{i+1}$ is a CPP over 
 $\mathbb{F}_{q^{d_i (d_{i+1}/d_{i})}} = \mathbb{F}_{q^{d_{i+1}}}$ if and only if $x^{q^{d_i} - 1} f_i(x)$ is a CPP over $\mathbb{F}_{q^{d_i}}$. Note that
\begin{equation*}
 x^{q^{d_i} - 1} f_i(x)|_{\mathbb{F}_{q^{d_i}}} 
 = 
 \begin{cases}
  f_i(x), \mbox{ if } i \geq 1 \mbox{ (since $f_i(0) = 0$ for each $i \geq 1$);} \\
  x^{q-1}f_0(x), \mbox{ if } i = 0.
 \end{cases}
 \end{equation*}
Denote $H(x) := x^{q-1}f_0(x)$.
We claim that $f_1$ is a CPP over $\mathbb{F}_{q^{d_1}}$ if and only if $f_0$ is a CPP over $\F$ satisfying $f_0(0) = 0$.
Indeed, if $f_0(0) = 0$, then $H|_{\F} = f_0$ implying that $f_1$ is CPP over $\mathbb{F}_{q^{d_1}}$ if and only if $f_0$ is a CPP over $\F$. 
On the other hand, if $f_0(0) \neq 0$, write $f_0(x) = A(x) + b$ for some $A \in \F[x]$ such that $A(0) = 0$ and some $b \in \F^*$. 
We need to show that $H$ is not a CPP over $\F$. On the contrary, suppose that $H$ is a CPP (and hence PP) over $\F$. Since $H(0) = 0$, 
it follows that $H$ permutes $\F^*$. But $H|_{\F^*} = f_0|_{\F^*}$. 
Then $f_0$ permutes $\F^*$. This in turn implies that $A$ permutes $\F^*$. 
Hence $-b \in A(\F^*)$ giving $f_0(e) = 0$ for some $e \in \F^*$. But then $H(e) = H(0) = 0$, a contradiction. The claim follows.
Now induction yields that for $1 \leq i \leq n$, $f_i$ is a CPP over $\mathbb{F}_{q^{d_i}}$ if and only if $f_0$ is a CPP over $\F$ satisfying $f_0(0) = 0$.
Next, assuming $f_0(0) = 0$, we claim that for $0 \leq i \leq n$, 
\begin{equation}\label{eqn: induction}
f_i(x) = x \left( \sum_{k=0}^{i-1}a_k\left( x^{p^r-1} - T_k^i(x)^{p^r-1} \right) + \frac{f_0 \left(T_0^i(x)   \right)}{T_0^i(x)} \right).	
\end{equation}
Proceed by induction on $0 \leq i \leq n$.
When $i = 0$ the claim is clear. Assume (\ref{eqn: induction}) holds for some $i < n$. Then by the transitivity
of the trace function, we get 
$$
\dfrac{f_i\left(T^{i+1}_i(x)\right)}{T^{i+1}_i(x)} = \sum_{k=0}^{i-1}a_k \left( T_i^{i+1}(x)^{p^r-1} - T_k^{i+1}(x)^{p^r-1}  \right) + \dfrac{f_0\left( T_0^{i+1}(x) \right)}{T_0^{i+1}(x)}.
$$
Thus, by the definition of $f_{i+1}$, we have
\begin{align*}
f_{i+1}(x) 
&=
x \Biggr(c_ix^{p^r-1} - c_i T_i^{i+1}(x)^{p^r-1} \\
& \hspace{5em} +\sum_{k=0}^{i-1}a_k \left( T_i^{i+1}(x)^{p^r-1} - T_k^{i+1}(x)^{p^r-1}   \right) + \dfrac{f_0\left( T_0^{i+1}(x)  \right)}{T_0^{i+1}(x) }     \Biggr)
\\
&=
x \Biggr(\sum_{k=0}^{i} a_k x^{p^r-1} - \sum_{k=0}^i a_k T_i^{i+1}(x)^{p^r-1} \\
&  \hspace{5em} + \sum_{k=0}^{i-1}a_k \left( T_i^{i+1}(x)^{p^r-1} - T_k^{i+1}(x)^{p^r-1}   \right) + \dfrac{f_0\left( T_0^{i+1}(x)  \right)}{T_0^{i+1}(x) }     \Biggr)\\
&= 
x \Biggr( \sum_{k=0}^{i} a_k x^{p^r-1} - \sum_{k=0}^i a_k T^{i+1}_k(x)^{p^r-1}  + \dfrac{f_0\left( T_0^{i+1}(x)  \right)}{T_0^{i+1}(x) }     \Biggr)\\
&=
x \Biggr( \sum_{k=0}^{i} a_k\biggr( x^{p^r-1}- T^{i+1}_k(x)^{p^r-1}  \biggr) + \dfrac{f_0\left( T_0^{i+1}(x)  \right)}{T_0^{i+1}(x) }  \Biggr),
\end{align*}
satisfying the expression in (\ref{eqn: induction}). The claim follows. Now it only remains to notice that $f = f_n$.
\end{proof}

Corollary \ref{cor: recursive} generalizes \cite[Corollary 4.4]{simona} and \cite[Corollary 2.3]{wu: even cpp}.

\section{Inverse of the complete mapping}

In this section we obtain the compositional inverse of the complete mapping of Theorem \ref{thm: improv}.
See Theorem \ref{thm: CPP inverse} and Corollary \ref{cor: CPP inverse} for this. To achieve
this, we make use of the result in Section 2 of the compositional inverses of linearized binomials permuting the 
kernel of the trace map. Since inverses of complete mappings are also complete mappings, this signifies the construction of a new, albeit complicated, class 
of complete mappings. 

We first introduce some notation. Let $f_y \in \F[x]$ be a polynomial with parameter $y \in \F$ and inducing an injective map on a subset $V$ of $\F$. Let $f_y|_V^{-1} \in \F[x]$ be the polynomial inducing
the inverse map of $f|_V$. Then, for any $g \in \F[x]$, we mean by $f_{g(x)}|_V^{-1}(x) \in \F[x]$ the polynomial obtained by substituting $y$ with $g(x) \in \F[x]$ in the expression for
$f_y|_V^{-1}(x)$.

\begin{lem}[{\bf Corollary 3.14, \cite{tuxanidy}}]\label{lem: inverse coulter}
 Using the same notations of Lemma \ref{lem: coulter} and assuming that $f$ permutes $\Fn$, the following
 two results hold:
 
 (i) If $p \mid n$ or $x \in \Fn$ is such that $\varphi_{\bar{f}^{-1}(T(x))}$ permutes $\F$,
 then $\varphi_{\bar{f}^{-1}(T(x))}$ permutes $\Fn$, and the preimage of $x$ under $f$
 is given by
 $$
 f^{-1}(x) = \varphi_{\bar{f}^{-1}(T(x))}^{-1}(x),
 $$
 where $\bar{f}^{-1} := \bar{f}|_{\F}^{-1}$.
 
 (ii) If $p \nmid n$, then the compositional inverse of $f$ over $\Fn$ is given by
 $$
 f^{-1}(x) = n^{-1} \bar{f}^{-1}(T(x)) + \varphi_{\bar{f}^{-1}(T(x))}|_{\ker(T)}^{-1}\left( x - n^{-1}T(x)  \right).
 $$
\end{lem}

Recall that $1/f := f^{q-2}$ for a polynomial $f$ if $\im(f) \subseteq \F$ when $f$ is viewed as a mapping.
\begin{thm}\label{thm: CPP inverse}
Assume that the polynomial $f$ of Theorem \ref{thm: improv} is a permutation polynomial over $\Fn$. Then $\bar{f}(x) := x G(x)$ is a permutation polynomial over $\F$ and $\varphi_y(x) := a x^{p^r} + x (G(y) - a y^{p^r-1})$
induces a permutation of $\ker(T)$ for each $y \in \F$. Let $\bar{f}^{-1} \in \F[x]$ denote the compositional inverse of $\bar{f}$ over $\F$ and define the polynomial
$$
C(x) := \bar{f}^{-1}\left(T(x)  \right)^{p^r-1} - a^{-1}G\left( \bar{f}^{-1}\left(T(x)    \right)   \right) \in \F[x].
$$
 Then the following three results hold:
 
 (i) If $x \in \Fn$ is such that $C(x) = 0$, then the preimage of $x$ under $f$ is given by 
 $$
 f^{-1}(x) = \left( \dfrac{x}{a}  \right)^{\frac{q^n}{p^r}}.
 $$

 (ii) Otherwise if $x \in \Fn$ is such that $N_{q|p^d}(C(x)) \neq 1$, then $N_{q|p^d}(C(x))^n \neq 1$, and the preimage of $x$ under $f$ is given by
$$
 f^{-1}(x) = \dfrac{N_{q|p^d}\left(C(x) \right)^n} {1 - N_{q|p^d}\left(C(x) \right)^n}
 \sum_{i=0}^{\frac{mn}{d}-1} C(x)^{-\frac{p^{(i+1)r} -1}{p^r-1}} \left(a^{-1}x  \right)^{p^{ir}}.
$$
 
 (iii) Otherwise the preimage of $x \in \Fn$ under $f$ is given by
 $$
  f^{-1}(x) = n^{-1}\left( \bar{f}^{-1}(T(x)) + \sum_{j=0}^{\frac{m}{d}-1}  C(x)^{-\frac{p^{(j+1)r} - 1}{p^r - 1}}
 \left(  a^{-1} \sum_{k=1}^{n - 1} k   x^{p^{km\frac{r}{d}}} \right)^{p^{jr}}\right).
$$

\end{thm}

\begin{proof}
Write $f(x) = ax^{p^r} - xg(T(x))$ where $g(x) := ax^{p^r-1} - G(x)$.
 As we have seen in the proof of Theorem \ref{thm: improv}, the fact that $f$ is a PP over $\Fn$ implies that $\bar{f}(x) := x G(x)$ is a PP over $\F$ and 
 $\varphi_y(x) := ax^{p^r} - xg(y) = a(x^{p^r} - xg(y)/a) = aL_{g(y)/a, r}(x)$ induces a permutation of $\ker(T)$ for each $y \in \F$. 
 Note that $\varphi_y|_{\ker(T)}^{-1}(x) = L_{g(y)/a, r}|_{\ker(T)}^{-1}(a^{-1}x)$, where $L_{c,r}(x) := x^{p^r} - cx$ for $c \in \F$. 
 Moreover $C(x) = g(\bar{f}^{-1}(T(x)))/a$.
 
 (i): If $x \in \Fn$ is such that $C(x) = 0$, then $\varphi_{\bar{f}^{-1}(T(x))}(x) = a L_{C(x), r}(x) = ax^{p^r}$ permutes $\Fn$. Now Lemma \ref{lem: inverse coulter} (i) gives
 $$
 f^{-1}(x) = \varphi_{\bar{f}^{-1}(T(x))}^{-1}(x) = L_{g(y)/a, r}|^{-1}(a^{-1}x) = \left( \dfrac{x}{a} \right)^{\frac{q^n}{p^r}}.
 $$
 
(ii) Otherwise if $x \in \Fn$ is such that $N_{q|p^d}(C(x)) \neq 1$, 
 then, by similar arguments to those in the proof of Corollary \ref{cor: InvKer2}, we get
 $$
 N_{q^n|p^d}\left( \dfrac{g\left(\bar{f}^{-1}(T(x))   \right)}{a}   \right) = N_{q|p^d}\left( \dfrac{g\left(\bar{f}^{-1}(T(x))   \right)}{a}   \right)^n  \neq 1.
 $$
 Then by Corollary \ref{cor: invLBP}, $L_{g(\bar{f}^{-1}(T(x)))/a, r}$, and hence $\varphi_{\bar{f}^{-1}(T(x))}$, permute $\Fn$. 
 Since $\varphi_y(\F) \subseteq \F$ for each $y \in \F$, necessarily $\varphi_{\bar{f}^{-1}(T(x))}$ permutes $\F$.
 Thus we can apply Lemma \ref{lem: inverse coulter} (i) to obtain
 $$
 f^{-1}(x) = \varphi_{\bar{f}^{-1}(T(x))}^{-1}(x) = L_{g(\bar{f}^{-1}(T(x)))/a, r}^{-1}\left(a^{-1}x\right).
 $$
 Substituting $c$ with $g(\bar{f}^{-1}(T(x)))/a$ in Corollary \ref{cor: invLBP}, we get
 $$
 f^{-1}(x) = \dfrac{N_{q^n|p^d}\left(\dfrac{g\left(\bar{f}^{-1}(T(x))\right)}{a}  \right)} {1 - N_{q^n|p^d}\left(\dfrac{g\left(\bar{f}^{-1}(T(x))\right)}{a}  \right)}\sum_{i=0}^{\frac{nm}{d}-1}
\left(\dfrac{g\left(\bar{f}^{-1}(T(x))\right)}{a}  \right)^{-\frac{p^{(i+1)r}-1}{p^r-1}}\left( a^{-1}x \right)^{p^{ir}}.
 $$
 Now the result follows from the fact that $N_{q^n|p^d}(y) = N_{q|p^d}(y)^n$ for any $y \in \F$.
 
 (iii) Here we can use Theorem \ref{thm: InvKer2} to obtain the inverse of $L_{g(\bar{f}^{-1}(T(x)))/a,r}|_{\ker(T)}$. Since $L_{g(\bar{f}^{-1}(T(x)))/a,r}$, 
 and hence $\varphi_{\bar{f}^{-1}(T(x))}$, do not permute $\Fn$, we apply Lemma \ref{lem: inverse coulter} (ii) instead. This gives
 \begin{align*}
 f^{-1}(x) 
 &= 
 n^{-1} \bar{f}^{-1}(T(x)) + \varphi_{\bar{f}^{-1}(T(x))}|_{\ker(T)}^{-1}\left( x - n^{-1}T(x)  \right)\\
 &=
 n^{-1} \bar{f}^{-1}(T(x)) + L_{g\left(\bar{f}^{-1}(T(x))\right)/a, r}|_{\ker(T)}^{-1}\left( a^{-1}\left( x - n^{-1}T(x)  \right) \right).
 \end{align*}
 For the sake of brevity denote $\bar{L}(z) := L_{g\left(\bar{f}^{-1}(T(x))\right)/a, r}|_{\ker(T)}^{-1}(z) \in \F[z]$. 
 Since $\bar{L}$ is a $p$-polynomial, the above becomes
$$
 f^{-1}(x) = n^{-1} \bar{f}^{-1}(T(x)) + \bar{L}\left( a^{-1} x \right)
 - \bar{L} \left(a^{-1}n^{-1} T(x)  \right).
$$
 By the definition of $\bar{L}$ and by Theorem \ref{thm: InvKer2} with $c = g\left(\bar{f}^{-1}(T(x))\right)/a$, we get
\begin{align*}
 \bar{L} \left(a^{-1}n^{-1}  T(x)  \right) 
 &= 
  \sum_{j=0}^{\frac{m}{d}-1} \left( \dfrac{g\left(\bar{f}^{-1}\left(T(x)\right)\right)}{a} \right)^{-\frac{p^{(j+1)r} - 1}{p^r - 1}} 
 \left(  n^{-2} \sum_{k=1}^{n - 1} k  a^{-1} T(x) \right)^{p^{jr}}\\
 &=
 \dfrac{n^{-1}(n-1)}{2}\sum_{j=0}^{\frac{m}{d}-1} \left( \dfrac{g\left(\bar{f}^{-1}\left(T(x)\right)\right)}{a} \right)^{-\frac{p^{(j+1)r} - 1}{p^r - 1}} \left( a^{-1}T(x) \right)^{p^{jr}}.
 \end{align*}
Since $\bar{f}(y) = y G(y) = ay^{p^r} - yg(y)$ for any $y \in \F$, we have $y  = a \bar{f}^{-1}(y)^{p^r} - \bar{f}^{-1}(y)g(\bar{f}^{-1}(y))$. Then,
if we substitute $y$ with $T(x)$, the above becomes
\begin{align*}
 \bar{L} \left(a^{-1}n^{-1}  T(x)  \right) 
 &=
 \dfrac{n^{-1}(n-1)}{2}\sum_{j=0}^{\frac{m}{d}-1} \left( \dfrac{g\left(\bar{f}^{-1}(T(x))\right)}{a} \right)^{-\frac{p^{(j+1)r} - 1}{p^r - 1}}\\
 & \hspace{5em}
 \cdot \left( \bar{f}^{-1}(T(x))^{p^r} - \bar{f}^{-1}(T(x))\dfrac{g\left(\bar{f}^{-1}(T(x))\right)}{a}  \right)^{p^{jr}}\\
 &=
 \dfrac{n^{-1}(n-1)}{2}\sum_{j=0}^{\frac{m}{d}-1} \Biggl( \bar{f}^{-1}(T(x))^{p^{(j+1)r}} \left( \dfrac{g\left(\bar{f}^{-1}(T(x))\right)}{a} \right)^{-\frac{p^{(j+1)r} - 1}{p^r - 1}}\\ 
 & 
 \hspace{10em} - \bar{f}^{-1}(T(x))^{p^{jr}} \left( \dfrac{g\left(\bar{f}^{-1}(T(x))\right)}{a} \right)^{-\frac{p^{jr} - 1}{p^r - 1}}\Biggr)\\
 &=
 \dfrac{n^{-1}(n-1)}{2}\bar{f}^{-1}(T(x)) \left(\left( \dfrac{g\left(\bar{f}^{-1}(T(x))\right)}{a} \right)^{-\frac{p^{mr/d} - 1}{p^r - 1}}  - 1  \right)\\
 &=
 0
 \end{align*}
because, by Lemma \ref{lem: norm identity} (ii) and by assumption,
$$
\left( \dfrac{g\left(\bar{f}^{-1}(T(x))\right)}{a} \right)^{\frac{p^{mr/d} - 1}{p^r - 1}} 
= \left( \dfrac{g\left(\bar{f}^{-1}(T(x))\right)}{a} \right)^{\frac{p^m - 1}{p^d - 1}} = 1.
$$
Thus
\begin{align*}
f^{-1}(x) &= 
n^{-1}\bar{f}^{-1}(T(x)) + L_{g\left(\bar{f}^{-1}(T(x))\right)/a, r}|_{\ker(T)}^{-1}\left(a^{-1}x  \right)\\
&=
n^{-1}\bar{f}^{-1}(T(x)) + \sum_{j=0}^{\frac{m}{d}-1} \left( \dfrac{g\left(\bar{f}^{-1}\left(T(x)\right)\right)}{a} \right)^{-\frac{p^{(j+1)r} - 1}{p^r - 1}} 
 \left(  n^{-1} \sum_{k=1}^{n - 1} k  a^{-1} x^{p^{km\frac{r}{d}}} \right)^{p^{jr}}.
\end{align*}
The result follows.
\end{proof}

\begin{cor}\label{cor: CPP inverse}
 If the polynomial $f$ in Theorem \ref{thm: improv} is a permutation polynomial over $\Fn$, then its compositional inverse over $\Fn$ is given by
 \begin{align*}
  f^{-1}(x) 
  = &
  \left( 1 - C(x)^{q-1}   \right)\left( \dfrac{x}{a}  \right)^{\frac{q^n}{p^r}}\\
 & + C(x)^{q-1}\left( N_{q|p^d}\left( C(x)   \right)   - 1 \right)^{p^d-1}\dfrac{N_{q|p^d}\left(C(x) \right)^n} {1 - N_{q|p^d}\left(C(x) \right)^n}\\
 &  \cdot \sum_{i=0}^{\frac{mn}{d}-1} C(x)^{-\frac{p^{(i+1)r} -1}{p^r-1}} \left(a^{-1}x  \right)^{p^{ir}}\\
 & + \left( 1 - \left(N_{q|p^d}\left(C(x)\right)  -1 \right)^{p^d-1}\right)\\
 &  \cdot n^{-1}\left( \bar{f}^{-1}(T(x)) + \sum_{j=0}^{\frac{m}{d}-1}  C(x)^{-\frac{p^{(j+1)r} - 1}{p^r - 1}}
 \left(  a^{-1} \sum_{k=1}^{n - 1} k   x^{p^{km\frac{r}{d}}} \right)^{p^{jr}}\right).
 \end{align*}
\end{cor}

\begin{proof}
 We put the results of Theorem \ref{thm: CPP inverse} together. Note that this is a step function where only one of the three terms is non-zero at a time. 
 The first term of the expression is non-zero only when $C(x) = 0$ corresponding to the result in (i),
 while the second term is non-zero only when $C(x) \neq 0$ and $N_{q|p^d}(C(x)) \neq 1$ corresponding to (ii).
 The third term is non-zero otherwise; this corresponds to the result in (iii). These are the only possibilities for $C(x)$ and so we are done.
\end{proof}

Corollary \ref{cor: CPP inverse} generalizes \cite[Corollary 4.6]{tuxanidy} and \cite[Theorem 3.5]{wu: even cpp}.

\section{A class of mutually orthogonal Latin squares}

In this section we construct a class of mutually orthogonal Latin squares by the means of the class of complete mappings in Theorem \ref{thm: improv}.
We give the result in Theorem \ref{thm: MOLS}. Our class generalizes that of Theorem 5.5 in \cite{simona}.

Let us recall some definitions. 
A {\em quasigroup operation} $*$ on a set $G$ is a binary operation such that the equations 
$x * u = v$ and $u * y = v$ have a unique solution $x,y$ for every $u,v \in G$. We call $(G, *)$ a {\em quasigroup} of order $|G|$. The multiplication table of a quasigroup is called a {\em Latin square}.
A mapping $Q : G \times G \rightarrow G$ defines a quasigroup if the binary operation $u * v = Q(u,v)$ is a quasigroup operation on $G$. Two quasigroups $(G, *_1), (G, *_2)$ are {\em orthogonal} if the system
of equations $(x *_1 y, x*_2 y) = (r,s)$ has a unique solution $(x,y)$ for every $(r,s) \in G \times G$. 
In this case the multiplication tables of the two quasigroups are called {\em orthogonal Latin squares} (OLS).
A set of Latin squares such that each pair is mutually orthogonal is called a set of {\em mutually orthogonal Latin squares} (MOLS).

We need the following Theorem \ref{thm: sade} and Lemma \ref{lem: simona} dealing with fields of arbitrary characteristic. 
Theorem \ref{thm: sade} is due to Sade \cite{sade} and gives a method, known as the {\em diagonal method}, of constructing
Latin squares by using complete mappings. 
The case when the characteristic of the field is 2 was already considered in Theorem \ref{thm: sade} and Lemma \ref{lem: simona} of \cite{simona}, although the proofs are essentially identical.
We however add the proofs here for the convenience of the reader. 

\begin{thm}[{\bf \cite{sade}}]\label{thm: sade}
 Let $P$ be a complete mapping over $\F$. Then the mapping $Q : \F^2 \rightarrow \F$ given by
 $$
 Q(x,y) = P(x+y) + y
 $$
 defines a quasigroup that possesses at least one orthogonal mate, the group $(\F, +)$.
\end{thm}

\begin{proof}
 Since $P$ is a complete mapping, both $Q(u,y)$ and $Q(x,u)$ are permutations for each $u \in \F$. It follows that $Q$ defines a quasigroup. Now $(\F, Q)$ is orthogonal to $(\F, +)$ if and only if the equation
 $$
 (P(x+y) + y, x+y) = (r,s)
 $$
 has a unique solution $(x,y)$ for each $(r,s) \in \F^2$. This is equivalent to
 $$
 (x,y) = (P(s) + s - r, r - P(s))
 $$
 and so a unique solution exists (since $P$ is a complete mapping).
\end{proof}

The following lemma determines when two quasigroups, constructed in the same fashion as that of Theorem \ref{thm: sade}, are orthogonal to each other.

\begin{lem}\label{lem: simona}
 Let $P_1$ and $P_2$ be complete mappings over $\F$. Then the quasigroups corresponding to 
 \begin{align*}
  Q_1(x,y) &= P_1(x+y) + y\\
  Q_2(x,y) &= P_2(x + y) + y
 \end{align*}
are orthogonal if and only if $P_2 - P_1$ is a permutation over $\F$.
\end{lem}

\begin{proof}
 $(\F, Q_1)$ is orthogonal to $(\F, Q_2)$ if and only if 
 \begin{align*}
  P_1(x+y) + y &= r\\
  P_2(x+y) + y  &= s
 \end{align*}
has a unique solution $(x,y)$ for each $(r,s) \in \F^2$. The system is equivalent to
\begin{align*}
x &= P_1^{-1}(r-y) - y\\
x &= P_2^{-1}(s-y) - y.
\end{align*}
It follows that $(\F,Q_1)$ is orthogonal to $(\F, Q_2)$ if and only if 
$$
P_1^{-1}(r-y) = P_2^{-1}(s-y)
$$
has a unique solution for each $(r,s) \in \F^2$. This equation is equivalent to 
$$
r = y + P_1\circ P_2^{-1}(s-y)
$$
and hence to
$$
s - r = (s - y) -P_1 \circ P_2^{-1}(s- y) 
$$
which has a unique solution if and only if $I -P_1 \circ P_2^{-1} $ is a permutation over $\F$. Since $P_2$ is a permutation, the last occurs if and only if $P_2 - P_1$ is a permutation.
\end{proof}

We are now ready to give the construction of a set of mutually orthogonal Latin squares.

\begin{thm}\label{thm: MOLS}
 Let $q = p^m$ be a power of a prime number $p$ and let $n,r,$ be positive integers such that $(m,r) = (mn,r)$, $p \nmid n$ and $(n, p^{(m,r)} - 1) = 1$. 
 Let $\{b_i\}$ be a subset of $\F$ and let $G \in \F[x]$ such that each $\bar{f}_i(x) := x(G(x) + b_i)$ is a complete permutation polynomial over $\F$.
 For each such $b_i$, let $a_i \in \F^*$.
 Then the quasigroups, $(\Fn, Q_j)$, where
 $$
 Q_0(x,y) = x+y, \hspace{2em} Q_i(x,y) = P_i(x+y) + y
 $$
 with 
 $$
 P_i(x) = x\left(a_i x^{p^r-1} - a_iT_{q^n|q}(x)^{p^r - 1} + G(T_{q^n|q}(x)) + b_i \right), \hspace{1em} i \geq 1,
 $$
 form a set of $1 + |\{b_i\}|$ mutually orthogonal Latin squares of order $q^n$.
 \end{thm}
 
 \begin{proof}
  By Theorem \ref{thm: improv}, each $P_i$ is a CPP over $\Fn$ since each $\bar{f}_i$ is a CPP over $\F$ by assumption. Moreover, if $i \neq j$,
  $$
  \left(P_j - P_i\right)(x) = x\Bigl[  \left( a_j -a_i   \right) x^{p^r-1} -   \left( a_j -a_i   \right) T_{q^n|q}(x)^{p^r - 1} + (b_j - b_i) \Bigr]
  $$
  is a permutation over $\Fn$ by Corollary \ref{cor: improv} (since $b_j \neq b_i$). Now Lemma \ref{lem: simona} and Theorem \ref{thm: sade} imply the result.
 \end{proof}
 
 Theorem \ref{thm: MOLS} generalizes \cite[Theorem 5.5]{simona}.

 \section{A $p$-ary bent vectorial function from the Maiorana-McFarland class}
 
 As another application of the permutation class obtained in Theorem \ref{thm: improv}, in this section we construct a class of $p$-ary bent vectorial functions by the means of the
 Maiorana-McFarland class. Our result, given in Theorem \ref{thm: bent}, generalizes that of Theorem 5.9 in \cite{simona}.
 
 For a prime $p$, let $\xi_p = e^{2\pi \sqrt{-1}/p}$. The {\em Walsh transform} of a function $f : \mathbb{F}_{p^n} \rightarrow \mathbb{F}_p$ is the complex-valued function
 $\hat{f}$ defined by
 $$
 \hat{f}(b) = \sum_{x \in \Fpn} \xi_p^{f(x) + T_{p^n|p}(bx)}.
 $$
 The function $f$ is called a $p$-ary {\em bent function} if $|\hat{f}(b)| = \sqrt{p^n}$ for all $b \in \Fpn$ (see for example \cite{pott}). 
 Bent functions have the minimum correlation to the class of affine functions (or maximum {\em non-linearity} possible), an important concept in cryptography \cite{nyberg}. 
 The construction of bent functions have been a focus of attention in several works. See for instance \cite{pott, ribic, nyberg, simona, stanica}.
 It was noted by Nyberg in \cite{nyberg} that complete mappings are useful in the construction of bent functions.
 
 In general, a $(n,m)$-vectorial function $F = (f_0, f_1, \ldots, f_{m-1}) : \Fp ^n \rightarrow \Fp^m$ is called {\em $(n,m)$-bent} if any non-zero linear combination of its components
 is a bent function. An amply studied class of bent functions in the literature is the Maiorana-McFarland class (see the aforementioned citations).
 These are the functions $f : \Fpn^2 \rightarrow \Fp$ with form
 $$
 f(x,y) = T_{p^n|p}\left( x \pi(y) + g(y)  \right)
 $$
 where $\pi,g,$ are functions on $\Fpn$. The condition that $\pi$ is a permutation of $\Fpn$ is both necessary and sufficient for $f$ to be bent. 
 But in general we have the following:
 
 \begin{lem}[\cite{nyberg}]\label{lem: nyberg}
  The function $F = (f_0, f_1, \ldots, f_{m-1}) : \mathbb{F}_{p^n}^2 \rightarrow \mathbb{F}_{p}^m$, where each of the components $f_i$ is a Maiorana-McFarland function
  $$
  f_i(x,y) = T_{p^t|p}\left(x \pi_i(y) + g_i(y)   \right),
  $$
  is a $(2n, m)$-bent function if every non-zero linear combination over $\Fp$ of the functions $\pi_i$, where $i \in \{0,1, \ldots, m-1  \}$, is a permutation of $\mathbb{F}_{p^n}$.
 \end{lem}
 
The following application of the complete mapping of Theorem \ref{thm: improv} gives a new class of $p$-ary bent vectorial functions through the means of the Maiorana-McFarland class. 

 \begin{thm}\label{thm: bent}
  Let $q = p^m$ be a power of a prime number $p$ and let $n,r,$ be positive integers such that $(m,r) = (mn,r)$, $p \nmid n$ and $(n, p^{(m,r)} - 1) = 1$. 
  Let $S$ be a subspace of $\F$ over $\mathbb{F}_p$ of dimension $k \leq m$ and with some basis $\{\alpha, \alpha^p, \ldots, \alpha^{p^{k-1}} \}$. 
  Let $G \in \F[x]$ be such that $G(0) \not\in S\setminus\{0\}$ and $\bar{f}(x) := x(G(x) + b)$ is a permutation polynomial over $\F$ for each $b \in S \setminus \{0\}$. 
  For every $i \in \{0, 1, \ldots,k-1  \}$, let
  $$
  \pi_i(x) := x\left( a_i x^{p^r-1} - a_iT_{q^n|q}(x)^{p^r - 1} + G(T_{q^n|q}(x)) + \alpha^{p^i}  \right)
  $$
  for some $a_i \in \F$. Then the function $F = (f_0,f_1,\ldots,f_{k-1}) : \Fn^2 \rightarrow \mathbb{F}_p^k$, where
  $$
  f_i(x,y) := T_{q^n|p}\left( x \pi_i(y) + g_i(y)  \right),
  $$
  is a $(2mn, k)$-bent function.
 \end{thm}

 \begin{proof}
  By Lemma \ref{lem: nyberg} we need to show that every non-zero linear combination over $\mathbb{F}_p$ of the $\pi_i$'s is a permutation of $\Fn$.
  Let $(c_0,c_1,\ldots, c_{k-1}) \in \mathbb{F}_{p}^k \setminus\{0\}$ and note that $\sum_{i=0}^{k-1}c_i \alpha^{p^i} \neq 0$. 
  It follows that $x(\sum_{i=0}^{k-1}c_i G(x) + \sum_{i=0}^{k-1} c_i \alpha^{p^i})$ is a permutation of $\F$ (by our assumption on $\bar{f}$ as well).
  Then 
  $$
  \sum_{i=0}^{k-1}c_i \pi_i(x) = x\left( \sum_{i=0}^{k-1}c_ia_i x^{p^r-1} - \sum_{i=0}^{k-1} c_ia_i T_{q^n|q}(x)^{p^r-1} + \sum_{i=0}^{k-1} c_i G(T_{q^n|q}(x)) + \sum_{i=0}^{k-1} c_i \alpha^{p^i}  \right)
  $$
  is a permutation of $\Fn$ by Corollary \ref{cor: coulter} if $\sum_{i=0}^{k-1}c_i a_i = 0$, and by Theorem \ref{thm: improv} and Corollary \ref{cor: improv} otherwise. 
 \end{proof}

Theorem \ref{thm: bent} generalizes Theorem 5.9 in \cite{simona}.

\section{Conclusion}
In this paper we gave the compositional inverses of linearized binomials permuting the kernel of the trace map under the assumption that it does not permute the entire finite field in question.
Previously it had been obtained in \cite{wu} the compositional inverses of linearized permutation binomials. The significance of our result lies in its applications
to obtaining the compositional inverses of certain classes of permutation polynomials relying on the trace map, say the complete mapping of 
Theorem \ref{thm: improv}. The compositional inverse of this mapping was obtained in Section 4.
We also gave an improvement of a class of complete mappings 
recently given in \cite{wu and lin} as well as obtained a recursive construction of complete mappings involving multi-trace functions. 
Finally we used the improved complete mapping to construct a new class of mutually orthogonal Latin squares by means of the so called diagonal method, 
and to derive a $p$-ary bent
vectorial function from the Maiorana-McFarland class.

\end{document}